\documentclass{article}

\usepackage{amssymb,amsmath,color}
\usepackage{amsthm}
\usepackage{url}

\definecolor{red}{rgb}{1,0,0}
\definecolor{blue}{rgb}{.2,.2,.8}

\newtheorem{theorem}{Theorem}[section]
\newtheorem{lemma}{Lemma}
\newtheorem{corollary}[theorem]{Corollary}

\newtheorem{conjecture}{Conjecture}
\newtheorem{inequality}{Inequality}[section]

\theoremstyle{remark}

\begin{document}

\title{Truncated theta series and \\partitions into distinct parts}
\author{Mircea Merca\\
	\footnotesize Department of Mathematics, University of Craiova, 200585 Craiova, Romania\\
	\footnotesize Academy of Romanian Scientists, Ilfov 3, Sector5, Bucharest, Romania\\
	\footnotesize mircea.merca@profinfo.edu.ro
}
\date{}
\maketitle

\begin{abstract}
	Linear inequalities involving Euler’s partition function $p(n)$ have been the subject of recent studies. 
	In this article, we consider the partition function $Q(n)$ counting the partitions of $n$ into distinct parts.
	Using truncated theta series, we provide four infinite families of linear inequalities for $Q(n)$ and partition theoretic interpretations for these results.
\end{abstract}

\section{Introduction}

In this paper, in order to simplify the expressions, we denote  the $n$th triangular number by
$$T_n=n(n+1)/2$$
and the $n$th generalized pentagonal number by
$$G_n=T_n-T_{\lfloor n/2 \rfloor},$$
for any nonnegative integer $n$.

In \cite{Andrews12}, G.E. Andrews and M. Merca considered Euler's pentagonal number theorem
\begin{equation}\label{eq:1}
\sum_{n=0}^{\infty} (-1)^{T_n} q^{G_n} = (q;q)_\infty
\end{equation} 
and obtained the following truncated version:
\begin{equation}\label{eq:2}
\frac{1}{(q;q)_\infty} \sum_{n=0}^{2k-1} (-1)^{T_n} q^{G_n}
= 1+(-1)^{k-1} \sum_{n=1}^{\infty} \frac{q^{\binom{k}{2}+(k+1)n}}{(q;q)_n} 
\begin{bmatrix}
n-1\\k-1
\end{bmatrix},
\end{equation} 
where 
\begin{align*}
& (a;q)_n = \begin{cases}
1, & \text{for $n=0$,}\\
(1-a)(1-aq)\cdots(1-aq^{n-1}), &\text{for $n>0$,}
\end{cases}\\
& (a;q)_\infty = \lim_{n\to\infty} (a;q)_n,\\
& 		
\begin{bmatrix}
n\\k
\end{bmatrix} 
=
\begin{bmatrix}
n\\k
\end{bmatrix}_{q}
=
\begin{cases}
0, & \text{ if $k<0$ or $k>n$}\\
\dfrac{(q;q)_n}{(q;q)_k(q;q)_{n-k}}, &\text{otherwise.}
\end{cases}
\end{align*}
From \eqref{eq:2}, they derived an infinite family of inequalities for Euler's partition function $p(n)$:
\begin{equation}\label{eq:3}
(-1)^{k-1}\sum_{j=0}^{2k-1} (-1)^{T_j} p(n-G_j) \geqslant 0,
\end{equation}
with strict inequality if $n>G_{2k}$.

Inspired by these results, V.J.W. Guo and J. Zeng \cite{Guo} considered two other classical theta identities, usually attributed to
Gauss \cite[p.23, eqs. (2.2.12) and (2.2.13)]{Andrews76}, i.e., 
\begin{equation}\label{eq:4}
1+2\sum_{n=1}^{\infty} (-q)^{n^2} = \frac {(q;q)_\infty} {(-q;q)_\infty},
\end{equation}
\begin{equation}\label{eq:5}
\sum_{n=0}^{\infty} (-q)^{T_j} = \frac {(q^2;q^2)_\infty} {(-q;q^2)_\infty},
\end{equation}
and proved new truncated forms of these: 
\allowdisplaybreaks{
	\begin{align}
	& \frac{(-q;q)_{\infty}}{(q;q)_{\infty}} \left(1+2\sum_{j=1}^{k} (-q)^{j^2} \right) \label{eq:6}\\
	& \qquad\qquad = 1+(-1)^k \sum_{n=k+1}^{\infty} \frac{(-q;q)_k (-1;q)_{n-k} q^{(k+1)n}}{(q;q)_n} 
	\begin{bmatrix}
	n-1\\k-1
	\end{bmatrix}, \notag \\
	& \frac{(-q;q^2)_{\infty}}{(q^2;q^2)_{\infty}} \sum_{j=0}^{2k-1} (-q)^{T_j} \label{eq:7}\\
	& \qquad\qquad = 1+(-1)^k \sum_{n=k}^{\infty} \frac{(-q;q^2)_k (-q;q^2)_{n-k} q^{2(k+1)n-k}}{(q^2;q^2)_n} 
	\begin{bmatrix}
	n-1\\k-1
	\end{bmatrix}_{q^2}, \notag
	\end{align}
}respectively. 
As an immediate consequence of \eqref{eq:6}, they deduced an infinite family of inequalities for the number of overpartitions of $n$, $\overline{p}(n)$:
\begin{equation}\label{eq:8}
(-1)^k \left(\overline{p}(n) +2 \sum_{j=1}^{k} (-1)^j \overline{p}(n-j^2)\right) \geqslant 0,
\end{equation}
with strict inequality if $n\geqslant (k+1)^2$. Similarly, from \eqref{eq:7} Guo and Zeng obtained an infinite family of inequalities for $pod(n)$, the  number of partitions of $n$ in which odd parts are not repeated:
\begin{equation}\label{eq:9}
(-1)^{k-1} \sum_{j=0}^{2k-1} (-1)^{T_j} pod \left( n-T_j \right)  \geqslant 0,
\end{equation}
with strict inequality if $n\geqslant (2k+1)k$.

Very recently, Andrews and Merca \cite{Andrews17} have revealed that \eqref{eq:2}, \eqref{eq:6} and \eqref{eq:7}
are essentially corollaries of the Rogers-Fine identity  \cite[p.15]{Rogers}:
\begin{equation*}
\sum_{n=0}^{\infty} \frac{(\alpha;q)_n \tau^n }{(\beta;q)_n}
= \sum_{n=0}^{\infty} \frac{(\alpha;q)_n (\alpha\tau q/\beta;q)_n \beta^n \tau^n q^{n^2-n}(1-\alpha\tau q^{2n})}{(\beta;q)_n (\tau;q)_{n+1}}.
\end{equation*}
A new infinite family of inequalities for Euler's partition function $p(n)$ is proved in this context: 
if at least one of $n$ and $k$ is odd,
\begin{equation}\label{eq:1.10}
(-1)^{k-1} \sum_{j=0}^{2k-1} (-1)^{T_j}  p(n-T_j)  \geqslant 0. 
\end{equation}

Other recent investigations on the truncated theta series  can be found in several papers by Chan, Ho and Mao \cite{Chan},
Chern \cite{Chern}, He, Ji and Zang \cite{He},   Kolitsch \cite{Kolitsch},  Kolitsch and Burnette \cite{KolitschB}, Mao \cite{Mao,Mao17}, Merca \cite{Merca16}, 
Merca, Wang and Yee \cite{MWY},
Wang and Yee \cite{WY1,WY2}, and Yee \cite{Yee}.

In this paper, motivated by these results, we shall provide similar families of inequalities for the partition function $Q(n)$ which counts the number of partition of $n$ into distinct parts. We consider $Q(0)=1$ and $Q(n)=0$ for all negative $n$.

By the truncated pentagonal number theorem \eqref{eq:2}, with $q$ replaced by $q^2$, we easily deduce the following result.

\begin{theorem}\label{Th:1.1}
	For a positive integer $k$, we have
	\begin{align*}
	& (-q;q)_\infty \sum_{n=0}^{2k-1} (-1)^{T_n} q^{2G_n}\\
	& \qquad 	= \frac{(q^2;q^2)_\infty}{(q;q^2)_\infty}+(-1)^{k-1} \frac{(q^2;q^2)_\infty}{(q;q^2)_\infty} \sum_{n=1}^{\infty} \frac{q^{k(k-1)+2(k+1)n}}{(q^2;q^2)_n} 
	\begin{bmatrix}
	n-1\\k-1
	\end{bmatrix}_{q^2}.
	\end{align*}
\end{theorem}

An immediate consequence owing to the positivity of the sum on the right is given by the following inequality  where $\delta_{i,j}$ is the Kronecker delta function.

\begin{corollary}\label{Cor:1.2}
	For $m,n\geqslant 0$, $k\geqslant 1$,
	$$(-1)^{k-1} \left( \sum_{j=0}^{2k-1} (-1)^{T_j} Q\left(n-2G_j\right) - \delta_{n,T_m} \right)  \geqslant 0,$$
	with strict inequality if and only if  $n\geqslant 2G_{2k}$.		
\end{corollary}

As we can see in \cite[Corollary 4.5]{Merca}, the recurrence relation given by the limiting case $k\to\infty$ of this corollary is known:
\begin{equation}\label{rec:1}
\sum_{j=0}^\infty (-1)^{T_j} Q(n-2G_j) 
=\begin{cases}
1, & \text{if $n=T_m$, $m\in\mathbb{N}_0$,}\\
0, & \text{otherwise}.
\end{cases}
\end{equation}
The number of terms in this recurrence relation is about $\sqrt{4n/3}$. 
By this recurrence relation, we can deduce the following parity result.

\begin{corollary}\label{Cor:1.3}
	Let $n$ be a positive integer.
	The number of representations of $n$ as the sum of a generalized pentagonal number and a twice generalized pentagonal number is odd if and only if $n$ is a triangular number.
\end{corollary}

In analogy with \eqref{eq:6}, we have a stronger result than Theorem \ref{Th:1.1}. 

\begin{theorem}\label{Th:1.4}
	For the positive integers $k$ and $r$, we have:
	\begin{align*}
	& (-q;q)_\infty \left( 1+2\sum_{j=1}^k (-1)^j q^{r\cdot j^2} \right)\\
	& \quad = \frac{(-q;q)_\infty (q^r;q^r)_\infty}{(-q^r;q^r)_\infty} + 2(-1)^k q^{r(k+1)^2} \frac{(q^r;q^{2r})_\infty}{(q;q^2)_\infty}
	\sum_{j=0}^{\infty}\frac{q^{(2k+2j+3)rj}}{(q^{2r};q^{2r})_j(q^r;q^{2r})_{k+j+1}}.
	\end{align*}
\end{theorem}

Considering Euler's pentagonal number theorem \eqref{eq:1} and the following two identities
\begin{align*}
& \frac{(-q;q)_\infty (q^2;q^2)_\infty}{(-q^2;q^2)_\infty} =  \sum\limits_{j=0}^\infty (-1)^{T_{\lfloor j/2 \rfloor}} q^{G_j},\\
& \frac{(-q;q)_\infty (q^3;q^3)_\infty}{(-q^3;q^3)_\infty} = \sum\limits_{j=0}^\infty q^{G_j},
\end{align*}
the positivity of 
$$
\sum_{j=0}^{\infty}\frac{q^{(2k+2j+3)rj}}{(q^{2r};q^{2r})_j(q^r;q^{2r})_{k+j+1}}
$$
allows us to deduce the following families of inequalities for the partition function $Q(n)$.

\begin{corollary}\label{cor:1.5}
	For $m,n\geqslant 0$, $k\geqslant 1$,
	\begin{enumerate}
		\item[a)] $\displaystyle{ (-1)^{k}\left( Q(n)+2 \sum_{j=1}^{k} (-1)^j Q(n-j^2)-(-1)^{T_m} \delta_{n,G_m}\right) \geqslant 0, }$
		\item[] with strict inequality if and only if  $n\geqslant (k+1)^2$.
		\item[b)] $\displaystyle{(-1)^{k}\left( Q(n)+2 \sum_{j=1}^{k} (-1)^j Q(n-2j^2)-(-1)^{T_{\lfloor m/2 \rfloor}}\delta_{n,G_m}\right) \geqslant 0},$
		\item[] with strict inequality if and only if  $n\geqslant 2(k+1)^2$.
		\item[c)] $\displaystyle{(-1)^{k}\left( Q(n)+2 \sum_{j=1}^{k} (-1)^j Q(n-3j^2)-\delta_{n,G_m}\right) \geqslant 0},$
		\item[] with strict inequality if and only if  $n\geqslant 3(k+1)^2$.	
	\end{enumerate}
\end{corollary}

There is a substantial amount of numerical evidence to conjecture that the sums in this corollary satisfy the following inequalities.

\begin{conjecture}\label{C:1}
	For $m,n\geqslant 0$, $k\geqslant 1$,
	\begin{align*}
	& (-1)^{k}\left( Q(n)+2 \sum_{j=1}^{k} (-1)^j Q(n-j^2)-(-1)^{T_m} \delta_{n,G_m}\right) \\
	& \qquad \geqslant (-1)^{k}\left( Q(n)+2 \sum_{j=1}^{k} (-1)^j Q(n-2j^2)-(-1)^{T_{\lfloor m/2 \rfloor}}\delta_{n,G_m}\right) \\
	& \qquad \geqslant (-1)^{k}\left( Q(n)+2 \sum_{j=1}^{k} (-1)^j Q(n-3j^2)-\delta_{n,G_m}\right).
	\end{align*}
\end{conjecture}

As we can see in \cite[Corollary 4.6]{Merca}, the recurrence relation given by the limiting case $k\to\infty$ of Corollary \ref{cor:1.5}.a) is known:
\begin{equation}\label{rec:2}
Q(n)+2 \sum_{j=1}^{\infty} (-1)^j Q(n-j^2)
=\begin{cases}
(-1)^{T_m}, & \text{if $n=G_m$, $m\in\mathbb{N}_0$,}\\
0, & \text{otherwise}.	
\end{cases}
\end{equation}
The number of terms in this recurrence relation is about $\sqrt{n}$.

The limiting case $k\to\infty$ of Corollary \ref{cor:1.5}.b) gives a new linear recurrence relation more efficient than \eqref{rec:2}:
\begin{equation}\label{rec:3}
Q(n)+2 \sum_{j=1}^{\infty} (-1)^j Q(n-2j^2)
=\begin{cases}
(-1)^{T_{\lfloor m/2 \rfloor}}, & \text{if $n=G_m$, $m\in\mathbb{N}_0$,}\\
0, & \text{otherwise}.	
\end{cases}
\end{equation}
We see that the number of terms in this recurrence relation is about $\sqrt{n/2}$. 

Corollary \ref{cor:1.5}.c) provides a new and extremely efficient algorithm for computing the partition function $Q(n)$:
\begin{equation}\label{rec:4}
Q(n)+2 \sum_{j=1}^{\infty} (-1)^j Q(n-3j^2)
=\begin{cases}
1, & \text{if $n=G_m$, $m\in\mathbb{N}_0$,}\\
0, & \text{otherwise}.	
\end{cases}
\end{equation}
The number of terms in this recurrence relation is about $\sqrt{n/3}$. 

As consequence of the recurrence relations \eqref{rec:2}-\eqref{rec:4}, we remark the following parity result.

\begin{corollary}\label{cor:1.6}
	Let $n$ be a positive integer.
	\begin{enumerate}
		\item[a)] 	The number of representations of $n$ as the sum of a generalized pentagonal number and a square number is odd if and only if $n$ is a generalized pentagonal number.
		\item[b)] 	The number of representations of $n$ as the sum of a generalized pentagonal number and a twice square number is odd if and only if $n$ is a generalized pentagonal number.
		\item[c)] 	The number of representations of $n$ as the sum of a generalized pentagonal number and a thrice square number is odd if and only if $n$ is a generalized pentagonal number.
	\end{enumerate}
\end{corollary}

Theorems \ref{Th:1.1} and \ref{Th:1.4} are good reasons to look for new infinite families of linear inequalities for the partition function $Q(n)$.
The rest of this paper is organized as follows. We will first prove Theorem \ref{Th:1.4} in Section \ref{S2}. In Section \ref{S3}, we provide combinatorial interpretations for Corollaries \ref{Cor:1.2} and \ref{cor:1.5}. Two infinite families of linear inequalities for $Q(n)$ which we have experimentally discovered are introduced in Section \ref{S4}.

\section{Proof of Theorems \ref{Th:1.4}}\label{S2}

The result follows directly from

\begin{lemma}\label{L:2.1}
	For a positive integer $k$,
	$$\sum_{j=0}^\infty (-1)^j q^{j^2+2j(k+1)} = (q^{2k+3};q^2)_\infty \sum_{j=0}^\infty \frac{q^{j(2j+2k+3)}}{(q^2;q^2)_j(q^{2k+3};q^2)_j}.$$
\end{lemma}

\begin{proof}
	To prove the lemma, we consider the second identity by Heine's transformation  of ${_{2}}\phi_{1}$ series \cite[(III.2)]{gasper}, namely
	$$
	{_{2}}\phi_{1}\bigg(\begin{matrix}a, b\\ c\end{matrix}\,;q,z\bigg)
	= \frac{(c/b;q)_\infty(bz;q)_\infty}{(c;q)_\infty(z;q)_\infty} 
	{_{2}}\phi_{1}\bigg(\begin{matrix}abz/c, b\\ bz\end{matrix}\,;q,c/b\bigg).
	$$
	We have
	\begin{align*}
	&\sum_{j=0}^\infty (-1)^j q^{j^2+2j(k+1)}\\
	&\qquad = \lim_{\tau\to 0} \sum_{j=0}^\infty \frac{(q^{2k+3}/\tau;q^2)_j}{(\tau;q^2)_j}\tau^j\\
	&\qquad = \lim_{\tau\to 0} {_{2}}\phi_{1}\bigg(\begin{matrix}q^{2},  q^{2k+3}/\tau\\ \tau\end{matrix}\,;q^{2},\tau\bigg)\\
	&\qquad = \lim_{\tau\to 0} \frac{(\tau^2/q^{2k+3};q^2)_\infty(q^{2k+3};q^2)_\infty}{(\tau;q^2)^2_\infty} {_{2}}\phi_{1}\bigg(\begin{matrix}q^{2k+5}/\tau,  q^{2k+3}/\tau\\ q^{2k+3}\end{matrix}\,;q^{2},\tau^2/q^{2k+3}\bigg)\\
	&\qquad = (q^{2k+3};q^2)_\infty \lim_{\tau\to 0}  \sum_{j=0}^\infty \frac{(q^{2k+5}/\tau;q^2)_j(q^{2k+3}/\tau;q^2)_j}{(q^2;q^2)_j(q^{2k+3};q^2)_j}\left( \frac{\tau^2}{q^{2k+3}}\right)^j\\
	&\qquad = (q^{2k+3};q^2)_\infty \lim_{\tau\to 0} \sum_{j=0}^\infty \frac{(-1)^j q^{j(j+1)}(q^{2k+3}/\tau;q^2)_j}{(q^2;q^2)_j(q^{2k+3};q^2)_j}\tau^j\\
	&\qquad = (q^{2k+3};q^2)_\infty \sum_{j=0}^\infty \frac{q^{j(2j+2k+3)}}{(q^2;q^2)_j(q^{2k+3};q^2)_j}
	\end{align*}	
	and the proof of this lemma is finished.
\end{proof}

With $q$ replaced by $q^r$, the Gauss identity \eqref{eq:4} becomes
$$1+2\sum_{n=1}^k (-1)^n q^{rn^2} = \frac{(q^r;q^r)_\infty}{(-q^r;q^r)_\infty}-2\sum_{n=k+1}^\infty (-1)^n q^{rn^2}.$$
Multiplying both sides of this identity by $(-q;q)_\infty$, we get
\allowdisplaybreaks{
	\begin{align*}
	&(-q;q)_\infty \left( 1+2\sum_{j=1}^k (-1)^j q^{rj^2} \right) \\
	&\quad = \frac{(-q;q)_\infty(q^r;q^r)_\infty}{(-q^r;q^r)_\infty} - 2(-q;q)_\infty \sum_{j=k+1}^{\infty} (-1)^j q^{rj^2}\\
	&\quad = \frac{(-q;q)_\infty(q^r;q^r)_\infty}{(-q^r;q^r)_\infty} + 2(-1)^k  \frac{q^{r(k+1)^2}}{(q;q^2)_\infty} \sum_{j=0}^{\infty} (-1)^j q^{rj^2+2rj(k+1)}\\
	&\quad =\frac{(-q;q)_\infty(q^r;q^r)_\infty}{(-q^r;q^r)_\infty} \\
	&\qquad + 2(-1)^k q^{r(k+1)^2} \frac{(q^{r(2k+3)};q^{2r})_\infty}{(q;q^2)_{\infty}}\sum_{j=0}^{\infty}\frac{q^{2rj^2+(2k+3)rj}}{(q^{2r};q^{2r})_j(q^{r(2k+3)};q^{2r})_j}
	\tag{By Lemma \ref{L:2.1}, with $q$ replaced by $q^2$.}\\
	&\quad =\frac{(-q;q)_\infty(q^r;q^r)_\infty}{(-q^r;q^r)_\infty}\\
	&\qquad + 2(-1)^k q^{r(k+1)^2} \frac{(q^{r};q^{2r})_\infty}{(q;q^2)_{\infty}(q^r;q^{2r})_{k+1}}\sum_{j=0}^{\infty}\frac{q^{(2k+2j+3)rj}}{(q^{2r};q^{2r})_j(q^{r(2k+3)};q^{2r})_j}\\
	&\quad =\frac{(-q;q)_\infty(q^r;q^r)_\infty}{(-q^r;q^r)_\infty} + 2(-1)^k q^{(k+1)^2} 
	\frac{(q^{r};q^{2r})_\infty}{(q;q^2)_{\infty}} \sum_{j=0}^{\infty}\frac{q^{(2k+2j+3)rj}}{(q^{2r};q^{2r})_j(q^r;q^{2r})_{k+j+1}}
	\end{align*}
}
and the proof is finished.

\section{Combinatorial interpretations}
\label{S3}

Regarding the inequality \eqref{eq:3}, we recall the following
partition theoretic interpretation given by Andrews and Merca \cite[Theorem 1]{Andrews12}:
\begin{equation*}
(-1)^{k-1}\sum_{j=0}^{2k-1} (-1)^{T_j} p(n-G_j) = M_k(n),
\end{equation*}
where $M_k(n)$ is the number of partitions of $n$ in which
$k$ is the least integer that is not a part and there are
more parts $>k$ than there are $<k$. In \cite{Yee} has
given a combinatorial proof of this result. Considering the generating function for $M_k(n)$, i.e.,
$$\sum_{n=0}^{\infty} M_k(n) q^n = \sum_{n=1}^{\infty} \frac{q^{\binom{k}{2}+(k+1)n}}{(q;q)_n}
\begin{bmatrix}
n-1\\k-1
\end{bmatrix},
$$
Theorem \ref{Th:1.1} can be written as:
\begin{align*}
(-q;q)_\infty \sum_{n=0}^{2k-1} (-1)^{T_n} q^{2G_n}- \frac{(q^2;q^2)_\infty}{(q;q^2)_\infty}
= (-1)^{k-1} \frac{(q^2;q^2)_\infty}{(q;q^2)_\infty} \sum_{n=0}^{\infty} M_k(n) q^{2n}
\end{align*}
or
\begin{align*}
& \left( \sum_{n=0}^\infty Q(n) q^n\right)  \left( \sum_{n=0}^{2k-1} (-1)^{T_n} q^{2G_n}\right) - \sum_{n=0}^\infty q^{T_n} \\
& \qquad\qquad\qquad\qquad= (-1)^{k-1} \left( \sum_{n=0}^\infty q^{T_n}\right) \left(  \sum_{n=0}^{\infty} M_k(n) q^{2n}\right),
\end{align*}
where we have invoked the theta identity \eqref{eq:5}.  In this way, we derive the following identity.

\begin{corollary}\label{cor:3.1}
	For $m,n\geqslant 0$, $k\geqslant 1$,
	$$(-1)^{k-1} \left( \sum_{j=0}^{2k-1} (-1)^{T_j} Q\left(n-2G_j\right) - \delta_{n,T_m} \right)  
	=\sum_{j=0}^\infty M_k\left(n/2-T_j/2 \right),
	$$
	where $M_k(x)=0$ if $x$ is not an integer.
\end{corollary}

More explicitly, the right hand side of this identity can be rewritten as:
$$\sum_{j=0}^\infty M_k\left(n-T_j/2 \right) 
=\sum_{j=-\infty}^\infty M_k\big(n-j(4j-1) \big),
$$
and
$$\sum_{j=0}^\infty M_k\big((2n+1)/2-T_j/2 \big)  
=\sum_{j=-\infty}^\infty M_k\big(n-j(4j-3) \big).
$$

Very recently, Merca, Wang and Yee \cite{MWY} provided the following partition-theoretic interpretation of the sum in Theorem \ref{Th:1.4}.

\begin{theorem} \label{Th:3.2}
	For a fixed $k\geqslant 1$, 
	\begin{equation*}
	\sum_{n=0}^{\infty} {M_{o,k}}(n) q^n =q^{(k+1)^2} \sum_{j=0}^{\infty}\frac{q^{(2k+2j+3)j}}{(q^2;q^2)_j(q;q^2)_{k+j+1}},
	\end{equation*}
	where ${M_{o,k}}(n)$ counts partitions of $n$ into odd parts such that all odd numbers less than or equal to $2k+1$ occur as parts at least once and the parts below the $(k+2)$-Durfee rectangle in the $2$-modular graph are strictly less than the width of the rectangle.
\end{theorem}

For $r>1$, we denote by  $Q_{r}(n)$ the number of partitions of $n$ into distinct parts not congruent to $0$ modulo $r$. The generating function of $Q_{r}(n)$ is given by
$$\sum_{n=0}^\infty Q_r(n) q^n = \frac{(q^r;q^{2r})_\infty}{(q;q^2)_\infty}.$$
Considering Theorems \ref{Th:1.4} and \ref{Th:3.2}, we immediately deduced the following partition-theoretic interpretation of Corollary \ref{cor:1.5}.

\begin{corollary}\label{cor:3.2}
	For $m,n\geqslant 0$, $k\geqslant 1$,
	\begin{enumerate}
		\item[a)] $\displaystyle{ (-1)^{k}\left( Q(n)+2 \sum_{j=1}^{k} (-1)^j Q(n-j^2)-(-1)^{T_m} \delta_{n,G_m}\right) = 2M_{o,k}(n) },$
		\item[b)] $\displaystyle{ (-1)^{k}\left( Q(n)+2 \sum_{j=1}^{k} (-1)^j Q(n-2j^2)-(-1)^{T_{\lfloor m/2 \rfloor}}\delta_{n,G_m}\right)  }$
		\item[] $\displaystyle{\qquad = 2\sum_{j=0}^{\lfloor n/2 \rfloor} M_{o,k}(j) Q_{2}(n-2j)},$
		\item[c)] $\displaystyle{ (-1)^{k}\left( Q(n)+2 \sum_{j=1}^{k} (-1)^j Q(n-3j^2)-\delta_{n,G_m}\right)  }$
		\item[] $\displaystyle{\qquad = 2\sum_{j=0}^{\lfloor n/3 \rfloor} M_{o,k}(j) Q_{3}(n-3j)}.$
	\end{enumerate}
\end{corollary}

We can rewrite Conjecture \ref{C:1} in the following equivalent form.

\begin{conjecture}\label{Cj:2}
	For $n\geqslant 0$, $k\geqslant 1$,
	$$\sum_{j=0}^{\lfloor n/3 \rfloor} M_{o,k}(j) Q_{3}(n-3j) \leqslant
	\sum_{j=0}^{\lfloor n/2 \rfloor} M_{o,k}(j) Q_{2}(n-2j) \leqslant
	M_{o,k}(n).$$
\end{conjecture}

Related to this conjecture, we remark that $Q_3(n)$ counts the partitions of $n$ into odd parts in which no part appears more than twice. More details about this combinatorial interpretation can be found in \cite{Alladi,Andrews00}. On the other hand, $Q_2(n)$ counts the partitions of $n$ into distinct odd parts. For $n\geqslant 0$, it is clear that
$$Q_3(n)\geqslant Q_2(n).$$
This inequality makes Conjecture \ref{Cj:2} more interesting.

Combinatorial proofs of Corollaries \ref{cor:3.1} and \ref{cor:3.2} would be very interesting.

\section{Two open problems}
\label{S4}

Related to the sum in the right hand side  of the truncated pentagonal number theorem \eqref{eq:2}, we remark 
that there is a substantial amount of numerical evidence to state 
the following conjecture.

\begin{conjecture}
	For $k\geqslant 1$, the theta series
	$$(-q;q)_\infty \sum_{n=0}^\infty (-1)^{T_n} q^{G_{n+2k}} = (q^2;q^2)_\infty \sum_{n=1}^{\infty} \frac{q^{\binom{k}{2}+(k+1)n}}{(q;q)_n} 
	\begin{bmatrix}
	n-1\\k-1
	\end{bmatrix}$$
	has non-negative coefficients.
\end{conjecture}

Assuming this conjecture, we derive the following family of linear inequalities.

\begin{inequality}\label{I:4.1}
	For $m,n\geqslant 0$, $k\geqslant 1$,
	$$(-1)^{k-1} \left( \sum_{j=0}^{2k-1} (-1)^{T_j} Q\left(n-G_j\right) - (-1)^{T_m}\delta_{n,2G_m} \right) \geqslant 0$$
	with strict inequality if and only  if $n\geqslant G_{2k}$.
\end{inequality}

We notice that this inequality is equivalent with: for $n\geqslant 0$, $k\geqslant 1$,
$$\sum_{j=0}^\infty (-1)^{T_j} M_k(n-2G_j) \geqslant 0$$
with strict inequality if and only  if $n\geqslant G_{2k}$.

In \cite{Andrews17}, Andrews and Merca provide the following revision of \eqref{eq:7}:
\begin{align*}
& \frac{(-q;q^2)_\infty}{(q^2;q^2)_\infty} \sum_{j=0}^{2k-1}(-q)^{T_j}  \\
& \qquad = 1 - (-1)^k \frac{(-q;q^2)_k}{(q^2;q^2)_{k-1}} 
\sum_{j=0}^{\infty} \frac{q^{k(2j+2k+1)}(-q^{2j+2k+3};q^2)_{\infty}}{(q^{2k+2j+2};q^2)_{\infty}},
\end{align*}
with
$$\sum_{n=0}^\infty MP_k(n) q^n = \frac{(-q;q^2)_k}{(q^2;q^2)_{k-1}} 
\sum_{j=0}^{\infty} \frac{q^{k(2j+2k+1)}(-q^{2j+2k+3};q^2)_{\infty}}{(q^{2k+2j+2};q^2)_{\infty}},$$
where $MPk(n)$ is the number of partitions of $n$ in which the first part larger than $2k-1$ is odd and appears exactly $k$ times; all other odd parts appear at most once. A purely combinatorial proof of this result can be found in \cite{Ballantine}.
Related to the sum in the right hand side  of these identities, we remark 
the following conjecture.

\begin{conjecture}
	For $k\geqslant 1$, the theta series
	$$(-q;q)_\infty \sum_{n=0}^\infty (-1)^{T_n} q^{T_{n+2k}} = (q^4;q^4)_\infty \frac{(-q;q^2)_k}{(q^2;q^2)_{k-1}} 
	\sum_{j=0}^{\infty} \frac{q^{k(2j+2k+1)}(-q^{2j+2k+3};q^2)_{\infty}}{(q^{2k+2j+2};q^2)_{\infty}}$$
	has non-negative coefficients.
\end{conjecture}

Assuming this conjecture, we derive a new infinity family of linear inequalities.

\begin{inequality}\label{I:4.2}
	For $m,n\geqslant 0$, $k\geqslant 1$,
	$$(-1)^{k-1} \left( \sum_{j=0}^{2k-1} (-1)^{T_j} Q\left(n-T_j\right) - (-1)^{T_m}\cdot \delta_{n,4G_m} \right)  \geqslant 0,$$
	with strict inequality if and only if  $n\geqslant T_{2k}$.	
\end{inequality}

We remark that this inequality can be rewritten in terms of $MP_k(n)$ as follows: for $n\geqslant 0$, $k\geqslant 1$,
$$\sum_{j=0}^\infty (-1)^{T_j} MP_{k} (n-4G_j) \geqslant 0,$$
with strict inequality if and only if  $n\geqslant T_{2k}$.

The limiting case $k\to\infty$ of Identities \ref{I:4.1} and \ref{I:4.2} give the following linear recurrence relations:
\begin{align*}
&\sum_{j=0}^\infty (-1)^{T_j} Q(n-G_j) 
=\begin{cases}
(-1)^{T_{m}}, & \text{if $n=2G_m$, $m\in\mathbb{N}_0$,}\\
0, & \text{otherwise,}
\end{cases}\\
&\sum_{j=0}^\infty (-1)^{T_j} Q(n-T_j) 
=\begin{cases}
(-1)^{T_{m}}, & \text{if $n=4G_m$, $m\in\mathbb{N}_0$,}\\
0, & \text{otherwise.}
\end{cases}
\end{align*}

The effectiveness of these recurrence relations can not be called into  question because the number of terms in each relations is greater than $\sqrt{n/3}$.
As consequences of these recurrence relations, we can derive the following parity results.

\begin{corollary}
	Let $n$ be a positive integer.
	\begin{enumerate}
		\item[a)] The number of representations of $n$ as the sum of two generalized pentagonal numbers is odd if and only if $n$ is a twice generalized pentagonal number.
		\item[b)] The number of representations of $n$ as the sum of a generalized pentagonal number and a triangular number is odd if and only if $n$ is a four times generalized pentagonal number.
	\end{enumerate}
\end{corollary}

Finally, we remark that the linear recurrence relations presented in this paper can be easily derived considering the Jacobi triple product identity. 
It is still an open problem to give partition interpretations for our truncated sums in Inequalities \ref{I:4.1} and \ref{I:4.2}.




\bigskip


\end{document}